\documentclass[12pt,draft]{amsart}

\usepackage[all]{xy}
\usepackage{amsfonts}
\usepackage{amssymb}

\usepackage{enumerate}

\newtheorem{teo}{Theorem}
\newtheorem{lem}[teo]{Lemma}
\newtheorem{prop}[teo]{Proposition}

\newtheorem{cor}[teo]{Corollary}

\theoremstyle{definition}
\newtheorem{dfn}[teo]{Definition}

\def\<{\langle}
\def\>{\rangle}

\def\a{\alpha}
\def\b{\beta}
\def\g{\gamma}

\def\l{{\lambda}}

\def\G{{\Gamma}}

\def\A{{\mathcal A}}
\def\K{{\mathcal K}}
\def\M{{\mathcal M}}
\def\cN{{\mathcal N}}

\def\End{\mathop{\rm End}\nolimits}

\def\1{\mathbf 1}

\def\N{{\mathbb N}}

\begin{document}
\title[Locally adjointable operators]
{Locally adjointable operators on Hilbert $C^*$-modules}
\author{Denis Fufaev}
\author{Evgenij Troitsky}
\thanks{This work is
supported by the Russian Foundation for Basic Research
under grant 23-21-00097.}
\address{Moscow Center for Fundamental and Applied Mathematics,
Dept. Mech. and Math., 
	Lomonosov Moscow State University, 119991 Moscow, Russia}
\email{fufaevdv@rambler.ru}
\email{troitsky@mech.math.msu.su}


\begin{abstract}
In the theory of Hilbert $C^*$-modules over a $C^*$-algebra $\A$ (in contrast with the theory of Hilbert spaces) not each bounded operator ($\A$-homomorphism) admits an adjoint.
The interplay between the sets of adjointable and non-adjointable operators 
plays a very important role in the theory.
We study an intermediate notion of locally adjointable operator $F:\M \to \cN$, i.e.
such an operator that $F\circ \gamma$ is adjointable for any adjointable $\gamma:\A \to \M$.
We have introduced this notion recently and it has demonstrated its usefulness in the context of
theory of uniform structures on Hilbert $C^*$-modules. In the present paper we obtain an explicit description of locally adjointable operators in important cases.
\end{abstract}

\maketitle

\begin{dfn}
A (right) pre-Hilbert $C^*$-module over a $C^*$-algebra $\A$
is an $\A$-module equipped with
a sesquilinear form on the underlying linear space $\<.,.\>:\M\times\M\to \A$ such that
\begin{enumerate}
\item $\<x,x\> \ge 0$ for any $x\in\M$;
\item $\<x,x\> = 0$ if and only if $x=0$;
\item $\<y,x\>=\<x,y\>^*$ for any $x,y\in\M$;
\item $\<x,y\cdot a\>=\<x,y\>a$ for any $x,y\in\M$, $a\in\A$.
\end{enumerate}
A complete pre-Hilbert $C^*$-module w.r.t. its norm $\|x\|=\|\<x,x\>\|^{1/2}$ is called a \emph{Hilbert $C^*$-module}.

If a Hilbert $C^*$-module $\M$ has a countable subset which linear span is dense in $\M$, then it is called \emph{countably generated}.

By $\oplus$ we will denote the orthogonal direct sum of  Hilbert
$C^*$-modules.
\end{dfn}

We refer to \cite{Lance,MTBook,ManuilovTroit2000JMS} for the theory
of Hilbert $C^*$-modules.

\begin{dfn}\label{dfn:standard_hm}
The \emph{standard} Hilbert $C^*$-module $\ell^2(\A)$
is a sum of countably many copies of $\A$
with the inner product $\<a,b\>=
\sum_i a_i^*b_i$, where $a=(a_1,a_2,\dots)$ and $b=(b_1,b_2,\dots)$.
Denote by $\pi_k$, $k\in\N$, the projection $\pi_k: \ell^2(\A) \to \A$, $a \mapsto a_k$.

If $\A$ is unital, then $\ell^2(\A)$ is countably
generated.
\end{dfn}

This example of Hilbert $C^*$-modules is especially important due to
the Kasparov stabilization theorem:
for any countably generated Hilbert $C^*$-module $\M$
over any algebra $\A$,
there exists an isomorphism of Hilbert $C^*$-modules
(preserving the inner product)
$\M\oplus \ell^2(\A)\cong \ell^2(\A)$ \cite{Kasp} (see
\cite[Theorem 1.4.2]{MTBook}). 

\begin{dfn}
A bounded $A$-homomorphism $F:\M\to\cN$ of Hilbert $C^*$-modules is called 
\emph{operator}.
\end{dfn}

\begin{dfn}
For an operator $F:\M\to\cN$ on Hilbert $C^*$-modules over $\A$, we say that $F$ is
\emph{adjointable} with (evidently unique) \emph{adjoint operator} $F^*: \cN \to\M$ if
$\<Fx,y\>_{\cN}=\<x,F^*y\>_{\M}$ for any $x\in \M$ and $y\in \cN$.
\end{dfn}

The following notion was introduced in a particular case of functionals in \cite{FufTro2024UniformStr} and turned out very useful in the description of
$\A$-compact
operators in terms of uniform structures there (see also \cite{Troitsky2020JMAA} and \cite{TroitFuf2020} for the previous research).
\begin{dfn}
A bounded $\A$-morphism  $F:\M\to \cN$ of Hilbert $C^*$-modules is called
\emph{locally adjointable} if, for any adjointable morphism $\gamma: \A \to \cN$,
the composition $F\circ \g: \A \to \cN$ is adjointable.
\end{dfn}

All these definitions are applicable in the case $\cN=\A$. In this case
bounded $A$-operators are called ($\A$)-functionals, adjointable operators are called
adjointable functionals and locally adjointable operators are called locally adjointable functionals.
These sets are denoted by $\M'$, $\M^*$ and $\M'_{LA}$, respectively. 
Evidently
$$
\M^* \subseteq \M'_{LA} \subseteq  \M'.
$$
They are right Banach modules (for the last set see Theorem \ref{teo:Banach} below)
with respect to the action $(fa)(x) =a^* f(x)$, where $f\in \M'$, $x\in \M$, $a\in \A$.
Typically $\M'$ is not a Hilbert $C^*$-module (see \cite{Manuilov2023MathNachr,ManTroit2023}
for a recent progress in the field).

The following notion was introduced and studied in \cite{BakGul2004} and applied to the frame theory in \cite{AramBa} (with developments in \cite{Fuf2021faa} and \cite{Fuf2022Path}). In \cite{Bak2019} explicit results for $\ell_2(\A)$ were obtained. 
Denote by $LM(\A)$, $R(\A)$, and $M(\A)$ \emph{left}, \emph{right}, and (two-sided) \emph{multipliers} of algebra $\A$, respectively (the usual reference is \cite{Ped}, 
see also \cite{MTBook}). 
For any Hilbert $\A$-module $\cN$  a Hilbert $M(\A)$-module $M(\cN)$ (which is called the \emph{multiplier module} of $\cN$) containing $\cN$ as an ideal submodule associated with $\A$, i.e. $\cN=M(\cN)\A $ was defined in  \cite{BakGul2004}. Namely, $M(\cN)$ is the space of all
adjointable maps from $\A$ to $\cN$ being a Hilbert $C^*$-module over $M(\A)$ with the inner
product $\<r_1,r_2\>=r_1^*r_2$. This is really a multiplier because  
$\<r_1,r_2\>a=r_1^*r_2(a)\in \A$. This is an essential extension of $\cN$ in sense of 
 \cite{BakGul2004}. 
 
Any (modular) \emph{multiplier} $m\in M(\cN)$ represents an $\A$-functional $\widehat{m}$ on $\cN$ by the formula $\widehat{m}(x)=\<m,x\>$. This functional is adjointable and its adjoint is given by the formula $\widehat{m}^*(a)=ma$. In fact this map gives rise to an identification of $M(\cN)$ and the module $\cN^*$ of
adjointable functionals on $\cN$ (see, \cite{BakGul2004,Bak2019}), in particular,
\begin{equation}\label{eq:isom1}
(\ell^2(\A))^*\cong M(\ell^2(\A)).
\end{equation}

In \cite[Theorem 2.3]{Bak2019} the following isomorphism was obtained (we write it keeping in mind the difference between left and right modules):
\begin{equation}\label{eq:isom2}
(\ell^2(\A))'\cong \ell^2_{strong}(RM(\A)),
\end{equation}
where the last module is formed by all sequences $\G_i \in RM(\A)$ such that the series
$\sum_i \G_i^* \G_i $  is strongly convergent in $B(H)$ (assuming that $\A$ is faithfully and non-degenerately represented on Hilbert space $H$).

Below in Lemma \ref{lem:descrLA} we will prove some ``intermediate variant'' of these isomorphisms \eqref{eq:isom1} and \eqref{eq:isom2}: 
\begin{equation}\label{eq:isom3}
(\ell^2(\A))'_{LA}\cong (M(\ell^2(\A)))'. 
\end{equation}

Now we pass to results of the present paper.

\begin{lem}
A bounded $\A$-morphism  $F:\K \to \cN$ of Hilbert $C^*$-modules is 
adjointable if and only if, for any $y \in \cN$, the morphism $F_y:\K \to \A$,
$F_y(x)=\<y, F(x)\>$ is adjointable.
\end{lem}

\begin{proof}
Suppose that $F$ is adjointable. Then, for $x\in\K$, $a\in \A$
$$
\<F_y(x),a\>=\<y, F(x)\>^* a =\<F^*(y),x\>^* a = \<x, F^*(y)a \>, \qquad (F_y)^*(a)=F^*(y)a,
$$
and $F_y(x)$ is adjointable.

Conversely, suppose that each $F_y$ is adjointable. Then, for an approximate unit $\{u_\l\}$
in $\A$, one has
$$
\<F(x),y\>u_\l =\<y,F(x)\>^* u_\l = F_y(x)^* u_\l =\<F_y(x), u_\l \>_\A =\<x,  (F_y)^* u_\l\>_\K.
$$
Since
$$
\|(F_y)^* (u_\l-u_\mu)\|=\sup_{z\in\K,\: \|z\|\le 1} |\<z,  (F_y)^* (u_\l-u_\mu)\>|
=\sup_{z\in\K,\: \|z\|\le 1} |\<F_y(z),  u_\l-u_\mu\>_\A|
$$
$$
=\sup_{z\in\K,\: \|z\|\le 1} |\<y,F(z)\>^*  (u_\l-u_\mu)|
=\sup_{z\in\K,\: \|z\|\le 1} |\<F(z), y (u_\l-u_\mu)\>_\cN|\le \|F\| \cdot \| y (u_\l-u_\mu)\|,
$$
we obtain (see \cite[Lemma 1.3.8]{MTBook})
that  the net 
$(F_y)^* u_\l$ is a Cauchy net. So we can define an operator $G$ by $G(y)=\lim\limits_\l(F_y)^* u_\l$. The operator $G$ is evidently bounded by its defining formula.
Since the above limit is in norm topology, for any $x\in\K$, $y\in\cN$, we have 
$$
\<F(x),y\>=\lim\limits_\l\<F(x),y\>u_\l=\lim\limits_\l(\<y,F(x)\>)^*u_\l=\lim\limits_\l\<F_y(x),u_\l\>=
\lim\limits_\l\<x,(F_y)^*u_\l\>=
$$
$$
=\left\<x,\lim\limits_\l(F_y)^*u_\l\right\>=\<x,G(y)\>,
$$
so, $F$ is adjointable.
\end{proof}

\begin{cor}
A bounded $\A$-morphism  $F:\M\to \cN$ of Hilbert $C^*$-modules is 
locally adjointable if and only if, for any adjointable morphism $\gamma: \A \to \cN$ and any
$y \in \cN$, the morphism $F_{\g,y}: \A \to \A$, $F_{\g,y}(a)=\<y, F\circ \g (a)\>$ is
adjointable.
\end{cor}

\begin{teo}\label{teo:Banach}
Locally adjointable operators from  $\M$ to $\cN$ form a Banach subspace of the Banach 
space of all bounded $\A$-morphisms from  $\M$ to $\cN$.

In particular, locally adjointable endomorphisms of $\M$ form a Banach subalgerbra of the algebra 
$\End_\A(\M)$ of all bounded $\A$-endomorphisms.
\end{teo}

\begin{proof}
Indeed, if $\{F_n\}$ is a sequence of locally adjointable morphisms and $F_n\to F$ in norm, then for any adjointable morphism $\g$ we have that 
$\|F_n\circ\g-F\circ\g\|\le\|F_n-F\|\cdot\|\g\|$, so $F_n\circ\g\to F\circ\g$ in norm too and $F\circ\g$ is adjointable.
\end{proof}

\begin{prop}\label{prop:conj}
The dual module $(M(\ell_2(\A)))'$ of $M(\ell_2(\A))$ consists of all sequences $\a_i \in M(\A)$ such that
\begin{enumerate}
\item[\rm 1)] the partial sums of $\sum_i \a^*_i \a_i$ are bounded, i.e. this series is strong convergent in $B(H)$;
\item[\rm 2)] the series $\sum_i \a_i^* \b_i$ is left strict convergent for any $\b = \{\b_i\} \in M(\ell_2(\A))$;
\item[\rm 3)] its limit belongs to $M(\A) \subseteq LM(\A)$.
\end{enumerate}
\end{prop}
\begin{proof}
Suppose, $\a\in (M(\ell_2(\A)))'$, $\a: M(\ell_2(\A)) \to M(\A)$. 
Then its restriction on the submodule $\ell_2(M(\A))$ defines (by \cite[Theorem 2.3]{Bak2019}) a sequence $\a_i \in M(\A)$
which must satisfy the property 1). It also can be restricted to $\ell_2(\A)=M(\ell_2(\A))\A$, and also by \cite[Theorem 2.3]{Bak2019}
the action is given by 
\begin{equation}\label{eq:act_on_subm}
\sum_{i=1}^\infty \a^*_i \b_i a, \qquad \{\b_i\} \in M(\ell_2(\A)),\quad a\in \A, \mbox{ the series is norm-convergent}.
\end{equation}
This gives 2). 

Two left multipliers $u$ and $v$ coincide, if $ua=va$ for any $a\in \A$. Thus, the equality
$$
\a(\b)a = \a(\b a)=\sum_{i=1}^\infty \a^*_i \b_i a = \left(\sum_{i=1}^\infty \a^*_i \b_i \right)a
$$ 
implies 
\begin{equation}\label{eq:defining_functional}
\a(\b) = \sum_{i=1}^\infty \a^*_i \b_i
\end{equation}
and hence 3).

Also, (\ref{eq:defining_functional}) implies that the linear mapping $\a \mapsto \{\a_i\}$ is injective.

Conversely, if $\{\a_i\}$ satisfies 1)-3), then (\ref{eq:defining_functional}) defines an element 
of $(M(\ell_2(\A)))'$. Indeed, everything is evident, one needs only to verify that this $\a$ is bounded.
For any $m<n$ and $a\in \A$ we have
$$
\left(\sum_{i=m}^n \a^*_i \b_i \right)^*\sum_{i=m}^n \a^*_i \b_i \le
\left\| \sum_{i=m}^n \a^*_i \a_i \right\| \cdot \left\| \sum_{i=m}^n \b^*_i \b_i \right\|.
$$  
Hence, $\a$ is bounded, and the mapping is surjective.
\end{proof}

\begin{lem}\label{lem:descrLA}
An $\A$-functional $\G:\ell_2(\A) \to \A$ is locally adjointable if and only if its 
collection of coefficients $\G_i$ determines an element of $(M(\ell_2(\A)))'$.
\end{lem}

\begin{proof}
An $\A$-functional is defined by a sequence $\G_i \in RM(\A)$ such that 
\begin{equation}\label{eq:conv_sum}
\sum_i \G_i^* \G_i \mbox{ strongly converges in }B(H)
\end{equation} 
(see \cite{Bak2019} and \eqref{eq:isom2} above). The action on $\a=(\a_1,\a_2,\dots)\in \ell_2(\A)$ 
is defined by $\G(\a)=\sum_i \G^*_i \a_i$ and the series is norm-convergent. 
Suppose that $\G$ is locally adjointable.
Consider an arbitrary adjointable morphism $\g: \A\to \ell_2(\A)$.
The set of these morphisms is isomorphic, on the one hand,  to the space $(\ell_2(\A))^*$ of adjointable $\A$-functionals, and on the other hand, to the module $M(\ell_2(\A))$ (see \cite{BakGul2004,Bak2019}).
Namely there exist (by \cite[Theorems 1.8 and 2.1]{BakGul2004})  $\g_i \in M(\A)$ such that $\sum_i \g_i^* \g_i$ is strictly convergent and 
\begin{equation}\label{eq:action}
\g(a)=(\g_1 a, \g_2 a, \dots), \qquad a \in \A.
\end{equation}
Then 
$$
\G\circ \g (a)= \sum_i \G_i^* \g_i a,
$$
where the series $\sum_i \G_i^* \g_i=\mu$ is convergent in left strict topology and defines
an element $\mu \in LM(\A)$. This gives property 2) of Proposition \ref{prop:conj}.
This morphism $\A\to \A$ has to be adjointable and hence we have 
$\sum_i \G_i^* \g_i \in M(\A)$. This gives 3)  of Proposition \ref{prop:conj}.
In particular, for $\g=(0,\dots,0,1_{M(\A)},0,\dots)$, we have that $\G^*_i$ is an adjointable left multiplier, i.e.
$\G_i \in M(\A)$. Together with (\ref{eq:conv_sum}) this gives 1)  of Proposition \ref{prop:conj}.

The converse is similar. 
Indeed, from 1) it follows that the sequence $\{\G_i\}$ defines an element of $(\ell^2(\A))'$ which acts by formula $\G(x)=\sum\limits_{i=1}^\infty\G_i^*x_i$, where series is norm-convergent. In particular, for any adjointable $\g:\A\to\ell^2(\A)$ and any $a\in \A$ we have $\G(\g(a))=\sum\limits_{i=1}^\infty\G_i^*\g_ia$.
From 2) it follows that $\left(\sum\limits_{i=1}^\infty\G_i^*\g_i\right)a=\sum\limits_{i=1}^\infty\G_i^*\g_ia=\G(\g(a))$, and from 3) it follows that $\left(\sum\limits_{i=1}^\infty\G_i^*\g_i\right)\in M(\ell^2(\A))$, i.e. $\sum\limits_{i=1}^\infty\G_i^*\g_i=\G\circ\g$ is adjointable. 
\end{proof}

The following statement will be used below and also seems to be of independent interest.

\begin{teo}
A bounded  $\A$-morphism $F:\M\to\ell_2(\A)$ is adjointable if and only if all  
projections $\pi_k\circ F$, $k\in\N$, are adjointable.
\end{teo}

\begin{proof}
If $F$ is adjointable then $\pi_k\circ F$ is adjointable since the projections $\pi_k$ are adjointable.

Suppose that for any projection $\pi_k$ we have that $\pi_k\circ F$ is adjointable.
Then, for any $y=(y_1,y_2,\dots)\in \ell^2(\A)$,
$$
\left\|\sum\limits_{k=p}^q(\pi_k\circ F)^*(\pi_k(y))\right\|=\sup\limits_{z\in \M,\: \|z\|\le 1} \left|\left\<z,\sum\limits_{k=p}^q(\pi_k\circ F)^*(\pi_k(y))\right\>_{\M}\right|=
$$
$$
=
\sup\limits_{z\in \M,\: \|z\|\le 1}\left|\left\<\sum\limits_{k=p}^q(\pi_k\circ F)(z),\pi_k(y)\right\>_{\A}\right|
=\sup\limits_{z\in \M,\: \|z\|\le 1}\left|\left\<F(z),\sum\limits_{k=p}^q\pi^*_k \pi_k(y)\right\>\right|
\le
$$
$$
\le\|F\|\cdot\left\|\sum\limits_{k=p}^q \pi^*_k \pi_k(y)\right\|=
\|F\| \cdot \left\|\sqrt{\sum\limits_{k=p}^q y_k^*y_k}\right\|
$$
Since the series $\sum\limits_{k=1}^\infty y_k^*y_k$ is norm-convergent, this implies that, for every $y\in \ell_2(\A)$, the series $\sum\limits_{k=1}^\infty(\pi_k\circ F)^*(\pi_k(y))$ is also norm-convergent in $\M$ and the equality
 $S(y)=\sum\limits_{k=1}^\infty(\pi_k\circ F)^*(\pi_k(y))$ defines a bounded $\A$-operator, $S:\ell^2(\A)\to \M$. Also, for any $x\in \M$, $y\in\ell^2(\A)$,
$$
\<F(x),y\>_{\ell^2(\A)}=\sum\limits_{k=1}^\infty\<\pi_k\circ F(x),\pi_k(y)\>_{\A}=\sum\limits_{k=1}^\infty\<x,(\pi_k\circ F)^*(\pi_k(y))\>=\<x,S(y)\>,
$$
so $F$ is adjointable with $S$ being the adjoint operator. 
\end{proof}

\begin{cor}
\begin{enumerate}
\item[\rm 1).] A bounded  $\A$-morphism $F:\M\to\ell_2(\A)$ is locally adjointable if and only if all its 
projections $\pi_k\circ F$, $k\in\N$, are locally adjointable.
\item[\rm 2).] An endomorphism $F\in \End_\A(\ell_2(\A))$ is locally adjointable if and only if its 
matrix rows belong to $M(\ell_2(\A))'$.
\end{enumerate}
\end{cor}

\begin{cor}
$M(\ell_2(\A)) \subset (\ell_2(\A))'_{LA}$.
\end{cor}

\begin{proof}
Indeed, $M(\ell_2(\A))=(\ell_2(\A))^*\subset (\ell_2(\A))'_{LA}$.
\end{proof}



\end{document}